\documentclass[a4paper,12pt]{article}
\usepackage{amsmath,amsthm,amssymb,amsfonts}
\usepackage{bbm,bm}
\usepackage{latexsym}
\usepackage{mathrsfs}
\usepackage{threeparttable}
\usepackage{tabularx}
\usepackage{booktabs}
\usepackage{graphicx}
\usepackage{color}
\usepackage{indentfirst}
\usepackage{geometry}
\usepackage{caption}
\usepackage{lineno}
\usepackage{abstract}
\usepackage{enumerate,mdwlist}
\usepackage[numbers,sort&compress]{natbib}  

\theoremstyle{plain}
\newtheorem{The}{Theorem}[section]    
\newtheorem{Lem}[The]{Lemma}
\newtheorem{Cor}[The]{Corollary}
\newtheorem{Pro}[The]{Proposition}

\theoremstyle{remark}
\newtheorem{Rem}[The]{Remark}

\theoremstyle{definition}

\numberwithin{equation}{section}

\makeatletter

\newcommand{\Rmnum}[1]{\expandafter\@slowromancap\romannumeral #1@}
\makeatother

\begin{document}
\title{Fractional matching preclusion number of graphs\thanks{This work is supported by NSFC (Grant No. 11371180).}}

\author{Ruizhi Lin and Heping Zhang\thanks{Corresponding author.}}
\date{{\small School of Mathematics and Statistics, Lanzhou University,
 Lanzhou, Gansu 730000, P.R. China}\\
{\small E-mails:\ linrzh08@lzu.edu.cn, zhanghp@lzu.edu.cn}}

\maketitle

\begin{abstract}
    Let $G$ be a graph with an even number of vertices. The matching preclusion number of $G$, denoted by $mp(G)$, is the minimum number of edges whose deletion leaves the resulting graph without a perfect matching. We introduced a $0$-$1$ linear programming which can be used to find matching preclusion number of graphs. In this paper, by relaxing of the $0$-$1$ linear programming we obtain a linear programming and call its optimal objective value as fractional matching preclusion number of graph $G$, denoted by $mp_f(G)$. We show $mp_f(G)$ can be computed in polynomial time for any graph $G$. By using perfect matching polytope, we transform it as a new linear programming whose optimal value equals the reciprocal of $mp_f(G)$. For bipartite graph $G$, we obtain an explicit formula for $mp_f(G)$ and show that $\lfloor mp_f(G) \rfloor$ is the maximum integer $k$ such that $G$ has a $k$-factor. Moreover, for any two bipartite graphs $G$ and $H$, we show $mp_f(G \square H) \geqslant mp_f(G)+\lfloor mp_f(H) \rfloor$, where $G \square H$ is the Cartesian product of $G$ and $H$.

    \vskip 0.1 in


    \vskip 0.1 in


    \medskip
\end{abstract}
\section{Introduction}
    In recent decades, many networks are proposed to serve as the topology of a large-scare parallel and distributed system. In practice, edge (link) failures may occur in a network, so it is important to consider networks with faulty elements. For measuring the robustness of interconnection networks under the condition of edge failure Brigham et al. \cite{matching_preclusion} first introduced the concept of matching preclusion.
    Let $G$ be a graph with an even number of vertices. A \emph{perfect matching} in a graph is a set of edges such that every vertex is incident with exactly one edge in this set. A set of edges $F$ of $G$ is called a \emph{matching preclusion set} if $G-F$ has no perfect matching, and the \emph{matching preclusion number} of $G$, denoted by $mp(G)$, is the cardinality of the matching preclusion set with smallest size.

    Since matching preclusion problem was proposed, it has been studied for many graphs, such as hypercube \cite{matching_preclusion}, $k$-ary $n$-cube \cite{k_ary_n_cube}, tori network \cite{product}, balanced hypercube \cite{balanced_hypercube}, folded Petersen cube \cite{folded_petersen_cube}, cube-connected cycle \cite{cube_connected_cycle} and pancake and burnt pancake graph \cite{pancake1,pancake2}. Furthermore, there are also some papers studying matching preclusion for some classes of graph, such as bipartite graph \cite{bipartite1,bipartite2}, regular graph \cite{regular,LP_of_MP}, vertex-transitive graph \cite{vertex_transitive} and Cartesian product of graphs \cite{product,LP_of_MP}. In complexity issue, M. Lacroix et al. \cite{NPC_MP} showed that matching preclusion problem is NP-complete even for bipartite graphs.

    Here we define some graph theoretical terms and notations first.
    Let $G$ be an undirected graph with the vertex set $V(G)$ and the edge set $E(G)$. For a vertex $v \in V(G)$, we define the set of all neighbour of $v$, the set of all edges incident with $v$ and degree of $v$ by $N_G(v)$, $\partial_G(v)$ and $d_G(v)$, respectively. Let $X$ and $Y$ be two vertex sets of $G$. We denote the set of edges of $G$ with one end in $X$ and the other end in $Y$ by $E_G(X,Y)$ and $e_G(X,Y)=|E_G(X,Y)|$. The set $E_G(X,V(G) \backslash X)$ is called the \emph{edge cut} of $G$ associated with $X$ and is denoted by $\partial_G(X)$, and we say it is \emph{trivial} if $|X|=1$ or $|V(G)\backslash X|=1$. An edge cut $\partial_G(X)$ is called an \emph{odd cut} if $|X|$ and $|V(G)\backslash X|$ is odd.
    Similarly, if $D$ is a directed graph and $v \in V(D)$, then we denote the set of all in-neighbour of $v$, the set of all out-neighbour of $v$, the set of all arcs with $v$ being its tail and the set of all arcs with $v$ being its head by $N^-_D(v)$, $N^+_D(v)$, $\partial^-_D(v)$ and $\partial^+_D(v)$, respectively.  If there is no ambiguity, then we can omit the subscripts of these notations.

    Let the linear space $\mathbb{R}^{E(G)}$ be the set of all column vectors whose entries are indexed by the edges of $G$ over real field $\mathbb{R}$. Every subset $S \subseteq E(G)$ can be described by its \emph{incidence vector}, an $|E(G)|$ long column vector, $\bm{q}^S = (\alpha_e \mid e \in E(G)) \in \mathbb{R}^{E(G)}$, where
    \begin{equation*}
        \alpha_e=\left\{
        \begin{aligned}
            &1, &\quad & \text{if } e \in S, \\
            &0, &{}    & \text{otherwise.}  \\
        \end{aligned}\right.
    \end{equation*}
    Let the vertex-edge incidence matrix of $G$ be $M_G$. Noting that for each $v \in V(G)$, there is a row in $M_G$ corresponding to incidence vector of $\partial_G(v)$, we denote this row vector by $\bm{\partial}_G(v)$.

    In \cite{LP_of_MP}, we introduced a $0$-$1$ linear programming for matching preclusion number of $G$. Let $G$ be a graph with an even number of vertices. We denote $\mathcal{M}(G)$ be the set consisting of its all perfect matchings and $\bm{y}$ be a vector in $\mathbb{R}^{E(G)}$. The following $0$-$1$ linear programming (MP) can be used to find the matching preclusion number of $G$.

    \vspace{0.5\baselineskip}
    \noindent (MP):
    \vspace{-1.29\baselineskip}
    \begin{alignat}{2}
        \min \quad        & \bm{1}^T \bm{y}                    &{}    & \notag \\ 
        \mbox{s.t.}\quad  & (\bm{q}^M)^T \bm{y} \geqslant 1,   &\quad & \mbox{for every } M \in \mathcal{M}(G) \label{MP1} \\
                          & y_e \in \{0,1\},                   &{}    & \mbox{for every } e \in E(G). \label{MP2}
    \end{alignat}

    We can see that Constraint \eqref{MP1} ensures that the edge set induced by $\bm{y}$ intersects every perfect matching in $\mathcal{M}(G)$. So we have the following proposition.
    \begin{Pro}[\cite{LP_of_MP}]
        \label{opt_MP}
        The optimal objective value of (MP) is equal to $mp(G)$.
    \end{Pro}

    In \cite{LP_of_MP}, by applying this $0$-$1$ linear programming on $r$-regular graph $G$ we showed that $mp(G)=r$ if and only if each non-trivial odd cut of $G$ has at least $r$ edges.

    Fractional graph theory is a new branch of graph theory and widely studied in recent years. There are two principal methods to convert graph concepts from integer to fractional. The first is to formulate the concepts as integer programs and then to consider the linear programming relaxation. The second is to make use of the subadditivity lemma. Using these two methods, many fractional graph concepts were proposed, such as fractional matching number, fractional chromatic number, fractional chromatic index and so on. Many further ideas and results on fractional graph theory can be found in \cite{fractional_graph_theory}. Inspired by it, we relax the Constraint \eqref{MP2} in (MP) and get a new linear programming, denoted by (FMP), as follows:

    \vspace{0.32\baselineskip}
    \noindent(FMP):
    \vspace{-1.70\baselineskip}
    \begin{alignat}{2}
        \min \quad        & \bm{1}^T \bm{y}                    &{}    & \notag \\
        \mbox{s.t.}\quad  & (\bm{q}^M)^T \bm{y} \geqslant 1,   &\quad & \mbox{for every } M \in \mathcal{M}(G), \label{FMP1} \\
                          & y_e \geqslant 0,                   &{}    & \mbox{for every } e \in E(G). \label{FMP2}
    \end{alignat}

    Then we define the optimal objective value of (FMP) by the \emph{fractional matching precluison number} of $G$, denoted by $mp_f(G)$.
    It follows from the definition that $mp_f(G) \leqslant mp(G)$ for any graph $G$. Then it is natural to consider which graphs satisfy that $mp_f(G) = mp(G)$ and how large the difference between $mp(G)$ and $mp_f(G)$ can be.

    Recently, Y. Liu and W. Liu \cite{fractional_matching_preclusion} introduced a distinct graphic parameter also called fractional matching preclusion number of any graph $G$, denoted by $fmp(G)$. However, their idea is different from this paper, and they define $fmp(G)$ as the minimum number of edges from $G$ whose deletion leaves the resulting graph with no fractional perfect matching. Furthermore, they gave some propositions of this parameter, and then studied it for complete graphs, Petersen graph and twisted cubes.

    Network flow theory is very useful in our proof, then we introduce some notations first. Let $D$ be a directed graph and $f$ be a real-valued function defined on $E(D)$. We denote the \emph{excess} of $f$ at $v$ by
    \begin{equation*}
        x_f(v)=\sum\limits_{w \in N^+_D(v)} f(v,w)-\sum\limits_{u \in N^-_D(v)} f(u,v).
    \end{equation*}
    Let $s$ (the \emph{source}) and $t$ (the \emph{sink}) be two distinguished vertices in $D$. Then we say that $f$ is an \emph{$s-t$ flow} if $x_f(v)=0$ for all $v\in V(D) \backslash \{s,t\}$ (\emph{conservation condition}). Let $l$ and $c$ be a non-negative real function defined on $E(D)$, which are called \emph{lower bound} and \emph{capacity} of arc, respectively. We say that an $s-t$ flow $f$ is \emph{feasible} if $0 \leqslant f(u,v) \leqslant c(u,v)$ for all $(u,v) \in E(D)$ (\emph{capacity constraint}) and call the $x_f(s)$ the \emph{value} of $f$. 
    An \emph{$s-t$ cut} is an outcut $\partial_D^+(X)$, such that $s \in X$ and $t \in V(D) \backslash X$. The capacity of a cut $C=\partial_D^+(X)$ is the sum of the capacities of its arcs, denoted by $cap(C)$.
    In addition, we say that $f$ is a \emph{circulation} if $x_f(v)=0$ for all $v\in V(D)$, furthermore, $f$ is feasible if $l(u,v)\leqslant f(u,v) \leqslant c(u,v)$ for all $(u,v) \in E(D)$. For convenience, if $A,B$ are two vertex sets of $V(D)$ and $a,b$ are two vertices in $D$, then we denote $c(A,B)=\sum\limits_{u \in A,v \in B}c(u,v)$, $c(a,B)=c(\{a\},B)$ and $c(A,b)=c(A,\{b\})$.

    The rest of this article is organized as follows.
    In section \ref{mp_general_graph}, for any graph $G$, we show $mp_f(G)$ can be computed in polynomial time, and introduce a new linear programming whose optimal value equals the reciprocal of $mp_f(G)$.
    In section \ref{mp_bigraph}, for any bipartite graph $G$, we obtain an explicit formula for $mp_f(G)$ and an optimal solution of (FMP), and show that $\lfloor mp_f (G) \rfloor$ is the maximum integer $k$ such that $G$ has a $k$-factor. Moreover, for any positive integer $t$, we give an example $G_t$ with $mp(G_t)=t+1$ and $mp_f(G_t)=2$. 
    In section \ref{mp_bigraph_product}, we show that for any two bipartite graphs $G$ and $H$, $mp_f(G \square H) \geqslant mp_f(G)+\lfloor mp_f(H) \rfloor$.

\section{General graph}
    \label{mp_general_graph}
    Recall that matching preclusion problem is NP-complete even for bipartite graphs. However, we can show fractional matching preclusion problem can be solved for any graph in polynomial time and our main tool is equivalence of optimization and separation which is shown by M. Gr\"{o}tschel et al. in \cite{eq_sep_opt}.

    \begin{The}[\cite{eq_sep_opt}]
        \label{separation_optimization}
        For any rational polyhedron, the optimization problem is polynomially solvable if and only if the separation problem is polynomially solvable.
    \end{The}

    We denote the polyhedron defined by Constraint \eqref{FMP1} and \eqref{FMP2} by $P$. So we can construct the separation problem corresponding to (FMP) as follows:
    \begin{quote}
        Given a rational vector $\bm{y} \in \mathbb{R}^{E(G)}$, either decide that $\bm{y} \in P$
        or, find a rational vector $\bm{w} \in \mathbb{R}^{E(G)}$ such that $\bm{w}^T\bm{x} > \bm{w}^T\bm{y}$ for all $\bm{x} \in P$.
    \end{quote}
    In order to solve this separation problem, we first verify Constraint \eqref{FMP2} for $\bm{y}$. If there is $f \in E(G)$ such that $y_f <0$, then we set $w_f=1$ and $w_e=0$ for all $e \in E(G) \backslash \{f\}$, so we have that $\bm{w}^T\bm{x} \geqslant 0 > \bm{w}^T\bm{y}$ for all $\bm{x} \in P$. 
    Thus, we suppose $\bm{y}$ satisfy Constraint \eqref{FMP2} and consider Constraint \eqref{FMP1}. We can regard the vector $\bm{y}$ as weights on edges of $G$, then according to the algorithm given in chapter 5.3 of \cite{combiatorial_book}, we can obtain the minimum weight perfect matching $M_0$ of $G$ in polynomial time with respect to vertex number of $G$. So if $(\bm{q}^{M_0})^T\bm{y} \geqslant 1$, then we can decide $\bm{y}$ satisfies Constraint \eqref{FMP1}, which means $\bm{y} \in P$, otherwise we set $\bm{w}=\bm{q}^{M_0}$ and have that $\bm{w}^T\bm{x} \geqslant 1 > \bm{w}^T\bm{y}$ for all $\bm{x} \in P$. Thus, we have the following lemma.
    \begin{Lem}
        \label{separation_FMP_polynomial_time}
        For any graph $G$, the separation problem corresponding to (FMP) can be solved in polynomial time.
    \end{Lem}
    So by Theorem \ref{separation_optimization}, we obtain the following result.
    \begin{The}
        \label{FMP_polynomial_time}
        For any graph $G$, $mp_f(G)$ can be computed in polynomial time.
    \end{The}

    For further study of fractional matching preclusion number, we need to construct a linear programming to compute it with more direct constraints, and our idea comes from perfect matching polytope. First, we introduce the perfect matching polytope of a graph. Let $\bm{v}_1,\ldots,\bm{v}_m$ be vectors in $\mathbb{R}^n$. Vector $\bm{v}=\lambda_1\bm{v}_1+\ldots+\lambda_m\bm{v}_m$,($\lambda_i \in \mathbb{R}$) is called a \emph{linear combination} of $\bm{v}_1,\ldots,\bm{v}_m$. A \emph{convex combination} is a linear combination with $\lambda_1+\cdots+\lambda_m=1$ and each $\lambda_i \geqslant 0$. The \emph{linear (convex) hull} of $\{\bm{v}_1,\ldots,\bm{v}_m\}$ is the set of all linear (convex) combinations of $\bm{v}_1,\ldots,\bm{v}_m$. The \emph{perfect matching polytope $PM(G)$} of a graph $G$ is the convex hull of incidence vectors of all perfect matchings in $G$. Edmonds \cite{perfect_matching_polytope} gave fundamental results to describe the perfect matching polytope. 
    \begin{The}[\cite{perfect_matching_polytope}]
        \label{polytope}
        The perfect matching polytope $PM(G)$ may be described by the following constraints:
        \begin{enumerate*}
            \item[$(i)$] $\bm{x} \geqslant 0$
            \item[$(ii)$] $\bm{\partial}_G(v) \bm{x}=1$, for every vertex $v$ in $G$
            \item[$(iii)$] $(\bm{q}^C)^T \bm{x} \geqslant 1$, for every non-trivial odd cut $C$ of $G$.
        \end{enumerate*}
    \end{The}
    Then for any graph $G$, we introduce a linear programming (LP) to compute $mp_f(G)$ without finding all perfect matchings of $G$. Let $\bm{b}$ be a vector in $\mathbb{R}^{E(G)}$. Then (LP) is defined as follows:

    \vspace{0.5\baselineskip}
    \noindent(LP):
    \vspace{-1.7\baselineskip}
    \begin{alignat}{2}
        \min \quad        & z                                &{}    & \notag \\
        \mbox{s.t.}\quad  & z-b_e \geqslant 0,               &{}    & \text{for every edge } e \in E(G) \notag \\
                          & \bm{\partial}_G(v) \bm{b}=1,     &\quad & \text{for every vertex } v \in V(G) \label{incident_constraint} \\
                          & (\bm{q}^C)^T \bm{b} \geqslant 1, &{}    & \text{for every non-trivial odd cut } C \text{ of } G \label{blossom_constraint} \\
                          & \bm{b} \geqslant 0.              &{}    & \label{non_negative_constraint}
    \end{alignat}
    Let $L(G)$ be the optimal objective value of (LP). Then we have following result.
    \begin{The}
        \label{mp_to_lp}
        Let $G$ be a graph with a perfect matching. Then $mp_f(G)=1/L(G)$. 
    \end{The}
    \begin{proof}
        First we give the dual of (FMP) as follows:

        \vspace{0.5\baselineskip}
        \noindent (DFMP):
        \vspace{-1.7\baselineskip}
        \begin{alignat*}{2}
            \max \quad        & \bm{1}^T \bm{x}                                            &{}    & \\
            \mbox{s.t.}\quad  & \sum_{M \in \mathcal{M}(G)}x_M\bm{q}^M \leqslant \bm{1}    &\quad & \\
                              & x_M \geqslant 0,                                           &{}    & \text{for every } M \in \mathcal{M}(G).
        \end{alignat*}
        Let $w=\sum\limits_{M \in \mathcal{M}(G)} x_M$. Since $\bm{x}=\bm{0}$ is not an optimal solution of (DFMP), we suppose $w \neq 0$. Let $x'_M=x_M/w$ and $\bm{b}=\sum\limits_{M \in \mathcal{M}(G)}x'_M\bm{q}^M$. Then we rewrite (DFMP) as follows:
        \begin{alignat}{2}
            \max \quad        & w                                                  &{}    & \notag \\
            \mbox{s.t.}\quad  & w\bm{b} \leqslant 1                                &{}    & \notag \\
                              & \sum_{M \in \mathcal{M}(G)}x'_M\bm{q}^M=\bm{b}     &\quad & \label{binPMG1} \\
                              & \sum_{M \in \mathcal{M}(G)}x'_M=1                  &{}    & \label{binPMG2} \\
                              & x'_M \geqslant 0,                                  &{}    & \text{for every } M \in \mathcal{M}(G). \label{binPMG3}
        \end{alignat}
        By Constraints \eqref{binPMG1}--\eqref{binPMG3}, we have $\bm{b} \in PM(G)$. So by Theorem \ref{polytope}, we transform (DFMP) into the following form,
        \begin{alignat}{2}
            \max \quad        & w                                &{}    & \notag \\
            \mbox{s.t.}\quad  & w\bm{b} \leqslant 1              &{}    & \label{w1}\\
                              & \bm{\partial}_G(v) \bm{b}=1,     &{}    & \text{for every vertex } v \in V(G) \notag \\
                              & (\bm{q}^C)^T \bm{b} \geqslant 1, &{}    & \text{for every non-trivial odd cut } C \text{ of } G \notag \\
                              & \bm{b} \geqslant 0,              &\quad & \notag
        \end{alignat}
        Since $w$ is only bounded in Constraint \eqref{w1} which is equivalent to $w \cdot \max\{b_e \mid e \in E(G)\} \leqslant 1$, we only need to compute minimum value of $\max\{b_e \mid e \in E(G)\}$ under the rest three constraints. So it remains to consider following programming.
        \begin{alignat}{2}
            \min \quad        & z                                &{}    & \notag \\
            \mbox{s.t.}\quad  & z=\max\{b_e \mid e \in E(G)\}    &{}    & \label{z1} \\
                              & \bm{\partial}_G(v) \bm{b}=1,     &{}    & \text{for every vertex } v \in V(G) \notag \\
                              & (\bm{q}^C)^T \bm{b} \geqslant 1, &{}    & \text{for every non-trivial odd cut } C \text{ of } G \notag \\
                              & \bm{b} \geqslant 0.              &\quad & \notag \\ 
        \intertext{We can convert it into a linear programming by replacing the Constraint \eqref{z1} with}
                              & z-b_e \geqslant 0                &{}    & \text{for every edge } e \in E(G) \tag{\ref{z1}$'$} \label{z1'}.
        \end{alignat}
        Noting that the resulting linear programming is (LP), we have that $w\cdot L(G) \leqslant 1$. Thus, the optimal objective value of (DFMP) is $1/L(G)$, so $mp_f(G)=1/L(G)$.
    \end{proof}

    For any graph $G$, Theorem \ref{mp_to_lp} means we can compute $mp_f(G)$ by solving (LP), whose constraints are related to odd cuts of $G$, rather than perfect matchings of $G$. Since a polynomial algorithm for minimum odd cut was given in \cite{odd_cut}, 
    we can also solve the separation problem corresponding to (LP) in polynomial time, which implies that (LP) is also polynomially solvable by Theorem \ref{separation_optimization}.

\section{Bipartite Graph}
    \label{mp_bigraph}
    Noting that the odd cuts in $G$ are numerous, the constraints of (LP) may be very complex. But for bipartite graphs, we can obtain some better results. For a bipartite graph, Birkhoff \cite{bipartite_perfect_matching_polytope} described its perfect matching polytope.
    \begin{The}[\cite{bipartite_perfect_matching_polytope}]
        \label{bipartite_polytope}
        If $G$ is a bipartite graph, then the perfect matching polytope $PM(G)$ may be described by the following constraints:
        \begin{enumerate*}
            \item[$(i)$] $\bm{x} \geqslant 0$
            \item[$(ii)$] $\bm{\partial}_G(v) \bm{x}=1$, for every vertex $v$ in $G$
        \end{enumerate*}
    \end{The}
    Inspired by it, we construct a simpler linear programming (BLP) to find $L(G)$,

    \vspace{0.5\baselineskip}
    \noindent(BLP):
    \vspace{-1.7\baselineskip}
    \begin{alignat}{2}
        \min \quad        & z                                &{}    & \notag \\
        \mbox{s.t.}\quad  & z-b_e \geqslant 0,               &{}    & \text{for every edge } e \in E(G) \notag \\
                          & \bm{\partial}_G(v) \bm{b}=1,     &\quad & \text{for every vertex } v \in V(G) \label{BLP_incidence_constraint} \\
                          & \bm{b} \geqslant 0,              &{}    & \label{BLP_non_negetive_constraint}
    \end{alignat}
    that is, (LP) without Constraint $\eqref{blossom_constraint}$. 
    Then we can show the following lemma.
    \begin{Lem}
        \label{optimal of BLP}
        If $G$ is a bipartite graph with a perfect matching, then the optimal objective value (BLP) is $L(G)$.
    \end{Lem}
    \begin{proof}
        Since $G$ is bipartite, by Theorem \ref{bipartite_polytope}, $\bm{b}$ satisfies Constraints \eqref{BLP_incidence_constraint}--\eqref{BLP_non_negetive_constraint} in (BLP) if and only if $\bm{b} \in PM(G)$. On the other hand, by Theorem \ref{polytope}, $\bm{b} \in PM(G)$ if and only if $\bm{b}$ satisfies Constraints \eqref{incident_constraint}--\eqref{non_negative_constraint} in (LP). Thus, $(z,\bm{b})$ is a feasible solution of (BLP) if and only if $(z,\bm{b})$ is a feasible solution of (LP), so this lemma holds.
    \end{proof}
    \begin{Rem}
        \label{bipartite_algorithm}
        In section \ref{mp_general_graph}, we have shown that for any graph $G$, $mp_f(G)$ can be computed in polynomial time by equivalence of optimization and separation, but the resulting algorithm do not appear to be efficient in practice. Here, if we suppose $G$ is bipartite, then by Theorem \ref{mp_to_lp} and Lemma \ref{optimal of BLP} we can compute $mp_f(G)$ by solving (BLP). Since (BLP) have only $|V(G)|+|E(G)|$ constraints and $|E(G)|$ variables, it implies an efficient algorithm to compute $mp_f(G)$ for any bipartite graph $G$.
    \end{Rem}

    Furthermore, we obtain an explicit expression of $mp_f(G)$ for any bipartite graph $G$, which plays an important role in studying the connection between $mp_f(G)$ and the existence of $k$-factor of $G$. To achieve this, we need Max-Flow Min-Cut Theorem and Hall's Theorem in the following.
    \begin{The}[Max-Flow Min-Cut Theorem, \cite{max_flow_min_cut}]
        \label{maxmin}
        Given a digraph $D$ with source $s$ and sink $t$, and capacity $c$ on $E(D)$. Then the maximum value of any feasible $s-t$ flow equals the minimum capacity of any $s-t$ cut.
    \end{The}
    \begin{The}[Hall's theorem, \cite{hall}]
        \label{hallthe}
        Let $G$ be a bipartite graph with bipartition $(A,B)$. Then $G$ has a matching of $A$ into $B$ if and only if and $|N(S)| \geqslant |S|$ for all $S \subseteq A$.
    \end{The}
    \begin{The}
        \label{mpstar}
        Let $G=G(A,B)$ be a bipartite graph with $|A| \leqslant |B|$. Then
        \begin{equation*}
            mp_f(G)=\min \left\{ \left. \frac{e(X,Y)}{|X|+|Y|-|A|} \right| X \subseteq A, Y \subseteq B, |X|+|Y|-|A|>0 \right\}.
        \end{equation*}
        Furthermore, for every $X \subseteq A$ and $Y \subseteq B$ such that $mp_f(G)=\frac{e(X,Y)}{|X|+|Y|-|A|}$, $\widetilde{\bm{y}}$ with
        \begin{equation*}
            \widetilde{y}_e=\left\{
            \begin{aligned}
                & \frac{1}{|X|+|Y|-|A|}  & &\text{if } e \in E(X,Y), \\
                & 0                      & &\text{otherwise.}
            \end{aligned}
            \right.
        \end{equation*}
        is an optimal solution of (FMP).
    \end{The}
    \begin{proof}
        First we suppose $G$ has no perfect matching. Then (FMP) has only Constraint \eqref{FMP2}, so we have $mp_f(G)=0$. On the other hand, by Theorem \ref{hallthe}, there exists $Y_0 \subseteq B$ with $|Y_0| > |N(Y_0)|$. Let $X_0=A \backslash N(Y_0)$. Then $e(X_0,Y_0)=0$ and $|X_0|+|Y_0|-|A|=|Y_0|-|N(Y_0)|>0$, so $\min \left\{ \left. \frac{e(X,Y)}{|X|+|Y|-|A|} \right| X \subseteq A, Y \subseteq B, |X|+|Y|-|A|>0 \right\}=0$. Furthermore, for every $X \subseteq A$ and $Y \subseteq B$ such that $\frac{e(X,Y)}{|X|+|Y|-|A|}=0$, we have $E(X,Y)=\emptyset$. So $\widetilde{\bm{y}}=\bm{0}$ is an optimal solution of (FMP).

        Next we suppose $G$ has a perfect matching. Then $|A|=|B|$. By Theorem \ref{mp_to_lp} and Lemma \ref{optimal of BLP}, we only need to solve (BLP). Now we construct a network flow to determine the optimal objective value of (BLP). Let $D$ be a digraph with $V(D)=V(G) \cup \{s,t\}$ and $E(D)=\{(s,a) \mid a \in A\} \cup \{(b,t) \mid b \in B\} \cup \{(a,b) \mid a \in A, b \in B, ab \in E(G)\}$, where $s$ and $t$ are the source and sink, respectively. We assign capacities to all arcs of $D$ as follows:
        \begin{equation*}
            c(u,v)=\left\{
            \begin{aligned}
                &1 , &\quad & \text{if $u=s,v \in A$ or $u \in B,v=t$,} \\
                &z , &{}    & \text{if $u \in A$ and $v \in B$.} \\
            \end{aligned}
            \right.
        \end{equation*}
        We claim that $(z,\bm{b})$ is a feasible solution of (BLP) if and only if $D$ has a feasible $s-t$ flow with value $|A|$.

        If $D$ has a feasible $s-t$ flow $f$ with value $|A|$, then $f(s,u)=f(v,t)=1$ for every $u \in A$ and $v \in B$. Let $b_{uv}=f(u,v) \geqslant 0$ for every edge $uv \in E(G)$. Then for every vertex $u \in A$, we have
        \begin{equation*}
            \bm{\partial}_G(u)\bm{b}=\sum\limits_{w \in N_G(u)} b_{uw}=\sum\limits_{w \in N^+_D(u)} f(u,w)=f(s,u)=1.
        \end{equation*}
        Similarly, every vertex $v \in B$ satisfies $\bm{\partial}_G(v)\bm{b}=1$. Moreover, by the capacity constraint, we can verify that $(z,\bm{b})$ satisfies the rest constraints of (BLP). So $(z,\bm{b})$ is a feasible solution of (BLP). 
        Conversely, if $(z,\bm{b})$ satisfies the constraints of (BLP), then we define a function $f$ on $E(D)$ as follows:
        \begin{equation*}
            f(u,v)=\left\{
            \begin{aligned}
                &1,      &\quad & \mbox{if $u=s,v \in A$ or $u \in B,v=t$,} \\
                &b_{uv}, &{}    & \mbox{if $u \in A$ and $v \in B$.} \\
            \end{aligned}
            \right.
        \end{equation*}
        Clearly, $f$ is a feasible $s-t$ flow with value $|A|$. Thus, our claim holds.

        Noting that $\partial^+_D(s)$ is an $s-t$ cut with value $|A|$, $(z,\bm{b})$ is a feasible solution of (BLP) if and only if every $s-t$ cut $C=\partial^+_D(R)$ of $D$ satisfies $cap(C) \geqslant |A|$ by Theorem \ref{maxmin}. We set $X=R \cap A$ and $Y=B \backslash R$. So we have $cap(C)=|A|-|X|+ze(X,Y)+|B|-|Y|$. If $e(X,Y)=0$, then $N_G(X) \subseteq B \backslash Y$. Since $G$ has a perfect matching, by Theorem \ref{hallthe} we have $|X| \leqslant |N_G(X)|$. Thus, $cap(C)=|A|+(|B|-|Y|-|X|)=|A|+(|B \backslash Y|-|X|) \geqslant |A|+(|N_G(X)|-|X|) \geqslant |A|$. If $e(X,Y)>0$, then $cap(C) \geqslant |A|$ if and only if $z \geqslant \frac{|X|-(|B|-|Y|)}{e(X,Y)}$. So $(z,\bm{b})$ is a feasible solution of (BLP) if and only if $z \geqslant \frac{|X|-(|B|-|Y|)}{e(X,Y)}$ for every $X \subseteq A$, $Y \subseteq B$ and $e(X,Y) \geqslant 0$. Thus, we have
        \begin{equation*}
            L(G)=\max \left\{ \left. \frac{|X|-(|B|-|Y|)}{e(X,Y)} \right| X \subseteq A, Y \subseteq B, e(X,Y)>0 \right\} >0,
        \end{equation*}
        and
        \begin{equation*}
            mp_f(G)=\min \left\{ \left. \frac{e(X,Y)}{|X|+|Y|-|A|} \right| X \subseteq A, Y \subseteq B, e(X,Y)>0, |X|+|Y|-|A|>0 \right\}.
        \end{equation*}
        If $|X|+|Y|-|A| > 0$, then $|X| > |B \backslash Y|$. Since $G$ has a perfect matching, we have $|N(X)| \geqslant |X| > |B \backslash Y|$ by Theorem \ref{hallthe}. Thus, we have $N(X) \cap Y \neq \emptyset$, which means $e(X,Y) > 0$. Then we can remove the constraint $e(X,Y)>0$ in last formula and obtain the result we need.

        Let $X \subseteq A$ and $Y \subseteq B$ be two vertex sets such that $mp_f(G)=\frac{e(X,Y)}{|X|+|Y|-|A|}$. For every $M \in \mathcal{M}(G)$, we have $|M \cap E(X,B \backslash Y)|+|M \cap E(A \backslash X,B \backslash Y)|=|B \backslash Y|$ and $|M \cap E(X,Y)|+|M \cap E(X,B \backslash Y)|=|X|$, which means $|M \cap E(X,Y)|=|X|-(|B \backslash Y|-|M \cap E(A \backslash X,B \backslash Y)|) \geqslant |X|-|B \backslash Y| = |X|+|Y|-|A|$. So we have $(\bm{q}^M)^T \widetilde{\bm{y}}=\frac{|M \cap E(X,Y)|}{|X|+|Y|-|A|} \geqslant 1$, then $\widetilde{\bm{y}}$ is a feasible solution of (FMP). Furthermore, noting that $\bm{1}^T \widetilde{\bm{y}}=mp_f(G)$, $\widetilde{\bm{y}}$ is an optimal solution of (FMP). Thus, this theorem holds.
    \end{proof}

    We observe that for a bipartite graph $G$, $mp_f(G)$ is closely related to existence of $k$-factor. Let $f$ be a non-negative integer-valued function defined on $V(G)$. An $f$-factor is a spanning subgraph $G'$ of $G$ such that $d_{G'}(v)=f(v)$ for all $v \in V(G)$. Furthermore, if $f(v)=k$ for all $v \in V(G)$, we say $G'$ is a $k$-factor. The following result obtained by Ore \cite{r_factor1}, and Folkman and Fulkerson \cite{r_factor2} gave the criterion for a bipartite graph to have an $f$-factor.
    \begin{The}[\cite{r_factor1,r_factor2}]
        \label{existence_r_factor}
        Let $G=G(A,B)$ be a bipartite graph and let $f$ be a non-negative integer-valued function on $V(G)$. Then $G$ has an $f$-factor if and only if 
        \begin{enumerate*}
            \item[$(i)$] $\sum\limits_{u \in A}f(u)=\sum\limits_{v \in B}f(v)$ and
            \item[$(ii)$] for all $X \subseteq A$ and $Y \subseteq B$, we have $\sum\limits_{x \in X}f(x)\leqslant e(X,Y)+\sum\limits_{y \in B \backslash Y}f(y)$.
        \end{enumerate*} 
    \end{The}
    So it is easy to see that a bipartite graph $G=G(A,B)$ has a $k$-factor if and only if $|A|=|B|$, and $k|X| \leqslant e(X,Y)+k(|B|-|Y|)$ for all $X \subseteq A$ and $Y \subseteq B$. Then we have the following corollary by Theorem \ref{mpstar}.
    \begin{Cor}
        \label{mpstar_k_factor}
        Let $G$ be a bipartite graph. If $k$ is the maximum integer such that $G$ has $k$-factor, then $k=\lfloor mp_f(G) \rfloor$.
    \end{Cor}

    \begin{Rem}
        Given a graph $G$ and a non-negative integer-valued function $f$ defined on $V(G)$, R. Anstee in \cite{f_factor_algorithm} introduced an algorithm to find an $f$-factor or show that none exists in polynomial time. Here, if $G$ is bipartite and $k$ is non-negative integer, then by Corollary \ref{mpstar_k_factor} and Remark \ref{bipartite_algorithm}, we can determine whether a bipartite graph has a $k$-factor by solving (BLP), which implies a new method to check whether a bipartite graph has a $k$-factor in polynomial time.
    \end{Rem}

    By Corollary \ref{mpstar_k_factor} we can find some classes of graphs with same matching preclusion number and fractional matching preclusion number, such as trees with an even number of vertices and regular bipartite graphs. 
    \begin{Cor}
        Let $T$ be a tree with an even number of vertices. Then $mp_f(T)=mp(T)$.
    \end{Cor}
    \begin{proof}
        If $T$ has no perfect matching, then $mp_f(T)=mp(T)=0$. Next we suppose $T$ has a perfect matching. So we have $\lfloor mp_f(T) \rfloor \geqslant 1$ by Corollary \ref{mpstar_k_factor}. Let $v$ be a vertex in $T$ with $d(v)=1$. Noting that $\partial(v)$ is a matching preclusion set of $T$, we have $mp(T) \leqslant |\partial(v)| =1$. Thus, $mp_f(T)=mp(T)=1$.
    \end{proof}

    By Hall's theorem, the following result can be easily shown and next corollary holds immediately by Corollary \ref{mpstar_k_factor}.
    \begin{The}[\cite{bipartite1}]
        \label{regular_bipartite}
        Let $G$ be an $r$-regular bipartite graph. Then the edges of $G$ can be partitioned into $r$ perfect matchings and $mp(G)=r$.
    \end{The}
    \begin{Cor}
        Let $G$ be an $r$-regular bipartite graph. Then $mp_f(G)=mp(G)=r$.
    \end{Cor}
    On the other hand, we can show the gap between $mp_f(G)$ and $mp(G)$ may be very large by the following example. For each positive integer $k$, we construct a graph $G_k$ (see Fig. \ref{figure_G3}) as follows:
    $V(G_k)=A_k \cup B_k \cup C_k \cup D_k$ where $A_k=\{a_1,\ldots,a_{2k}\}$, $B_k=\{b_1,\ldots,b_{2k}\}$, $C_k=\{c_1,\ldots,c_k\}$ and $D_k=\{d_1,\ldots,d_k\}$, and $E(G_k)=\{a_ib_i \mid 1 \leqslant i \leqslant 2k\} \cup \{a_ic_j,b_id_j \mid 1 \leqslant i \leqslant 2k, 1 \leqslant j \leqslant k\}$. The following theorem shows $mp(G_k)-mp_f(G_k)=k-1$.

\begin{figure}[htbp]
    \centering
    \includegraphics[width=0.45\textwidth]{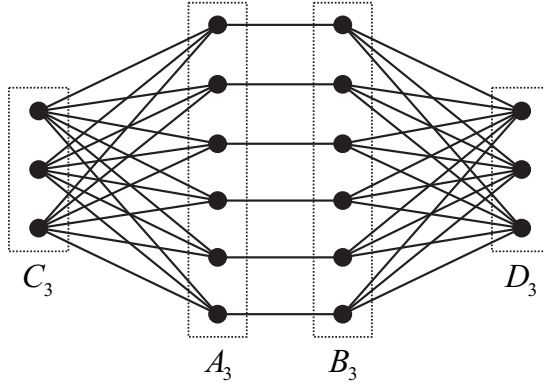}\\
    \caption{\footnotesize Graph $G_k$ ($k=3$).}\label{figure_G3}
\end{figure}

    \begin{The}
        Let $k$ be a positive integer. Then $mp(G_k)=k+1$ and $mp_f(G_k) = 2$.
    \end{The}
    \begin{proof}
        First we prove $mp(G_k)=k+1$. Noting that $\partial_{G_k}(a_1)$ is a matching preclusion set, we have $mp(G_k) \leqslant |\partial_{G_k}(a_1)|=k+1$. So it remains to show $mp(G_k)\geqslant k+1$. Next we show that for every vertex set $F$ with $|F| \leqslant k$, $G_k-F$ has a perfect matching.

        \textbf{Case 1.} $F \cap E(A_k,B_k)=\emptyset$. Then $F \subseteq E(A_k,C_k) \cup E(B_k,D_k)$. Suppose that $a_{2k}c_1 \in F$ without loss of generality. Noting that $H_1=G\left[\bigcup\limits_{i=1}^k\{a_i,c_i\}\right]$ and $H_2=G\left[\bigcup\limits_{i=1}^k\{b_i,d_i\}\right]$ are two $k$-regular complete bipartite graphs with $|E(H_1) \cap F| \leqslant k-1$ and $|E(H_2) \cap F| \leqslant k-1$, by Theorem \ref{regular_bipartite} we have that $H_1-F$ and $H_2-F$ have perfect matchings $M_1$ and $M_2$ respectively. Thus, $M_1\cup M_2 \cup \{a_ib_i \mid k+1 \leqslant i \leqslant 2k\}$ is a perfect matching of $G-F$.

        \textbf{Case 2.} $F \cap E(A_k,B_k) \neq \emptyset$. Then there exist at least $k$ edges in $E(A_k,B_k)\backslash F$ supposed to be $a_1b_1,a_2b_2,\ldots,a_kb_k$ without loss of generality. Then $G- \bigcup\limits_{i=1}^k\{a_i,b_i\}$ are two disjoint $k$-regular complete bipartite graphs, $H_1$ and $H_2$. Since $|E(H_1) \cap F| \leqslant k-1$ and $|E(H_2) \cap F| \leqslant k-1$, by Theorem \ref{regular_bipartite} we have that $H_1-F$ and $H_2-F$ have perfect matchings $M_1$ and $M_2$ respectively. Thus, $M_1\cup M_2 \cup \{a_ib_i \mid 1 \leqslant i \leqslant k\}$ is a perfect matching of $G-F$.

        Thus, we have $mp(G_k)=k+1$.
        On the the hand, $mp_f(G_k) \leqslant \frac{e(A_k,B_k)}{|A_k|+|B_k|-|V(G)|/2}=2$. Since $G_k$ has a $2$-factor consisting of $k$ disjoint cycles $c_ia_ib_id_ib_{i+k}a_{i+k}c_i$ where $1 \leqslant i \leqslant k$, we have $mp_f(G_k) \geqslant 2$ by Corollary \ref{mpstar_k_factor}. Thus, $mp_f(G_k)=2$.
    \end{proof}
\section{Cartesian Product of Bipartite Graphs}
    \label{mp_bigraph_product}
    In this section, we concentrate on Cartesian product of bipartite graphs. 
    The Cartesian product of two graphs $G$ and $H$ is a graph, denoted as $G \square H$, whose vertex set is $V(G) \times V(H)$, with two vertices $(g,h)$ and $(g',h')$ being adjacent if $g=g'$ and $hh' \in E(H)$, or $gg' \in E(G)$ and $h=h'$.

    Let $G$ and $H$ be two bipartite graphs with $V(G)=\{u_1,\ldots,u_n\}$, $E(G)=\{e_1,\ldots,e_m\}$, $V(H)=\{v_1,\ldots,v_p\}$ and $E(H)=\{f_1,\ldots,f_q\}$.
    We denote the incidence matrices of $G$ and $H$ by $M_G$ and $M_H$, respectively.
    Recall that the Kronecker product of two matrix $A$ and $B$, where $A=(a_{ij})$ is an $x \times y$ matrix, is defined by
    \begin{equation*}
        A \otimes B=
        \begin{pmatrix}
            a_{11}B & \cdots & a_{1y}B \\
            \vdots  & \ddots & \vdots  \\
            a_{x1}B & \cdots & a_{xy}B
        \end{pmatrix}.
    \end{equation*}
    Then the incidence matrix of $G \square H$ is $P=(I_p \otimes M_G, M_H \otimes I_n)$, where $I_k$ is $k \times k$ identity matrix. In order to compute $mp_f(G \square H)$, we should consider (BLP) for $G \square H$. So we rewrite the Constraint \eqref{BLP_incidence_constraint} of (BLP) in matrix form as $P\bm{b}=\bm{1}$ and denote the resulting linear programming by (BLP1). We accordingly represent $\bm{b}$ as
    \begin{equation*}
        \bm{b}=
        \begin{pmatrix}
            \bm{a}_1 \\ \vdots \\ \bm{a}_p \\ \bm{h}_1 \\ \vdots \\ \bm{h}_q
        \end{pmatrix},
        \text{where } \bm{a}_i=
        \begin{pmatrix}
            a_i^1 \\ \vdots \\ a_i^m
        \end{pmatrix}
        \text{and } \bm{h}_j=
        \begin{pmatrix}
            h_j^1 \\ \vdots \\ h_j^n
        \end{pmatrix}
        \text{for each } 1\leqslant i \leqslant p \text{ and } 1\leqslant j \leqslant q.
    \end{equation*}
    Let $G_i$ be the subgraph of $G \square H$ induced by $\{(a,v_i) \mid a \in V(G)\}$. Then we can see that $\bm{a}_i$ is an incidence vector of $G_i$, and if $f_j=v_xv_y$, then $\bm{h}_j$ indicates the edges between $V(G_x)$ and $V(G_y)$.
    \begin{Lem}
        \label{cartesian_product_pm}
        Let $G$ and $H$ be two bipartite graphs. If $H$ has a $k$-factor, then $G \square H$ has a $k$-factor.
    \end{Lem}
    \begin{proof}
        Let $F$ be a $k$-factor of $H$. Then the spanning subgraph $F'$ of $G \square H$ with edge set $E(F')=\bigcup\limits_{uv \in E(F)}\{(w,u)(w,v) \mid w \in V(G)\}$ is a $k$-factor of $G \square H$.
    \end{proof}
    \begin{Lem}
        \label{same_ah}
        Let $G$ and $H$ be two bipartite graphs. If $H$ is $r$-regular, then (BLP1) has an optimal solution such that $\bm{a}_1=\cdots=\bm{a}_p$ and $\bm{h}_1=\bm{h}_2=\cdots=\bm{h}_q$.
    \end{Lem}
    \begin{proof}
        Noting that a perfect matching of graph is also a $1$-factor of this graph,
        since $H$ has a perfect matching by Theorem \ref{regular_bipartite}, $G \square H$ has a perfect matching $M$ by Lemma \ref{cartesian_product_pm}. So $(1,\bm{q}^M)$ is a feasible solution of (BLP1), then (BLP1) has an optimal solution.
        Let $(\widetilde{z},\widetilde{\bm{b}})$ be an optimal solution of (BLP1) with $\widetilde{\bm{b}}=(\widetilde{\bm{a}}_1^T,\ldots,\widetilde{\bm{a}}_p^T,\widetilde{\bm{h}}_1^T,\ldots,\widetilde{\bm{h}}_q^T)^T$. Then $\widetilde{z}$ equals to some entry of $\widetilde{\bm{b}}$. So $\widetilde{z}=\max\limits_{1\leqslant k \leqslant m}\widetilde{a}_t^k$ for some $t \in \{1, \ldots ,p\}$ or $\widetilde{z}=\max\limits_{1 \leqslant k \leqslant n}\widetilde{h}_s^k$ for some $s \in \{1, \ldots ,q\}$. 

        If $\widetilde{z}=\max\limits_{1\leqslant k \leqslant m}\widetilde{a}_t^k$, then we set $\bm{a}_1=\cdots=\bm{a}_p=\widetilde{\bm{a}}_t$ and $\bm{h}_1=\cdots=\bm{h}_q=\frac{1}{r}\sum\limits_{f_l\in \partial_H(v_t)}\widetilde{\bm{h}}_l=\widetilde{\bm{h}}_{o}$. 
        So for each $(u_i,v_j)\in V(G \square H)$, we have
        \begin{align*}
            \bm{\partial}_{G \square H}((u_i,v_j))\bm{b} &= \bm{\partial}_G(u_i)\bm{a}_j + \sum\limits_{f_k \in \partial_H(v_j)}h_k^i \\
                                                         &= \bm{\partial}_G(u_i)\widetilde{\bm{a}}_t + \sum\limits_{f_k \in \partial_H(v_j)} (\frac{1}{r}\sum\limits_{f_l\in \partial_H(v_t)}\widetilde{h}_l^i) \\
                                                         &= \bm{\partial}_G(u_i)\widetilde{\bm{a}}_t + \sum\limits_{f_l\in \partial_H(v_t)}\widetilde{h}_l^i \\
                                                         &= \bm{\partial}_{G \square H}((u_i,v_t))\bm{\widetilde{b}} \\
                                                         &= 1,
        \end{align*}
        which means $P\bm{b}=\bm{1}$. Moreover, we have that for each $1\leqslant i \leqslant p$ and $1 \leqslant j \leqslant m$,
        \begin{equation*}
            a_i^j=\widetilde{a}_t^j \leqslant \widetilde{z}, 
        \end{equation*}
        and for each $1 \leqslant x \leqslant q$ and $1\leqslant y \leqslant n$,
        \begin{equation*}
            h_x^y=\frac{1}{r}\sum\limits_{f_l\in \partial_H(v_t)}\widetilde{h}_l^y \leqslant \max\limits_{f_l\in \partial_H(v_t)}\widetilde{h}_l^y \leqslant \widetilde{z}.
        \end{equation*}
        Then $(\widetilde{z},\bm{b})$ with $\bm{b}=(\widetilde{\bm{a}}_t^T,\ldots,\widetilde{\bm{a}}_t^T,\widetilde{\bm{h}}_o^T,\ldots,\widetilde{\bm{h}}_o^T)^T$ is a feasible solution of (BLP1). Furthermore, note that $(\widetilde{z},\widetilde{\bm{b}})$ is optimal, we have that $(\widetilde{z},\bm{b})$ is optimal too.

        Next we suppose $\widetilde{z}=\max\limits_{1 \leqslant k \leqslant n}\widetilde{h}_s^k$. We only need to denote one end of $f_s$ by $v_t$, and set $\widetilde{\bm{h}}_{o}=\frac{1}{r}\sum\limits_{f_l\in \partial_H(v_t)}\widetilde{\bm{h}}_l$ and $\bm{b}=(\widetilde{\bm{a}}_t^T,\ldots,\widetilde{\bm{a}}_t^T,\widetilde{\bm{h}}_o^T,\ldots,\widetilde{\bm{h}}_o^T)^T$. Then by a completely similar argument as above case, we can show $(\widetilde{z},\bm{b})$ is optimal.
    \end{proof}
    If $H$ is $r$-regular and $\bm{b}=(\bm{a}^T,\ldots,\bm{a}^T,\bm{h}^T,\ldots,\bm{h}^T)^T$, then $P\bm{b}=\bm{1}$ is equivalent to $M_G\bm{a}+r\bm{h}=\bm{1}$. So we define the linear programming (BLP2) as follows:

    \quad

    \noindent (BLP2):
    \vspace{-1.7\baselineskip}
    \begin{alignat*}{2}
        \min \quad        & z                            &{}    & \\
        \mbox{s.t.}\quad  & z-a_e \geqslant 0,           &{}    & \text{for every edge } e \in E(G) \\
                          & z-h_v \geqslant 0,           &{}    & \text{for every vertex } v \in V(G) \\
                          & M_G\bm{a}+r\bm{h}=\bm{1}     &\quad & {} \\
                          & \bm{a} \geqslant 0           &{}    & \\
                          & \bm{h} \geqslant 0.          &{}    &
    \end{alignat*}
    By Lemma \ref{same_ah}, the following corollary holds clearly.
    \begin{Cor}
        Let $G$ and $H$ be two bipartite graphs. If $H$ is $r$-regular, then (BLP1) and (BLP2) has the same optimal objective value, that is, $L(G \square H)$.
        \label{same_obj}
    \end{Cor}
    The following result on the existence of circulation due to Hoffman is useful in our proof.
    \begin{The}[Hoffman's Circulation Theorem, \cite{circulation_theorem}]
        \label{circulation_theorem}
        Given a digraph $D$ and lower bound $l$ and capacity $c$ on $E(D)$, there exists a feasible circulation if and only if every $R \subseteq V(D)$ satisfies
        \begin{equation*}
            c(V(D)\backslash R,R) \geqslant l(R,V(D) \backslash R).
        \end{equation*}
    \end{The}

    \begin{The}
        \label{mpstar_product}
        Let $G=G(A,B)$ be a bipartite graph with $|A|=|B|$ and $H$ be an $r$-regular bipartite graph. Then
        \begin{equation*}
            mp_f(G \square H)= \min \left\{ \left. \frac{e(X,Y)+r\min \{|X|,|Y|\}}{|X|+|Y|-|A|} \right| X \subseteq A, Y \subseteq B, |X|+|Y|-|A|>0 \right\}.
        \end{equation*}
    \end{The}
    \begin{proof}
        Since $H$ has a perfect matching by Theorem \ref{regular_bipartite}, $G \square H$ has a perfect matching by Lemma \ref{cartesian_product_pm}. So by Theorem \ref{mp_to_lp}, Lemma \ref{optimal of BLP} and Corollary \ref{same_obj}, we only need to solve (BLP2). Now we construct a circulation to determine the optimal objective value of (BLP2). Let $D$ be a digraph with $V(D)=V(G) \cup \{s,t\}$ and $E(D)=\{(s,a) \mid a \in A\} \cup \{(b,t) \mid b \in B\} \cup \{(a,b) \mid a \in A, b \in B, ab \in E(G)\} \cup \{(t,s)\}$. Then we assign lower bound and capacity as follows:
        \begin{align*}
            c(u,v)&=\left\{
            \begin{aligned}
                &1,   &\quad & \text{if $u=s,v \in A$ or $u \in B,v=t$,} \\
                &z,   &{}    & \text{if $u \in A$ and $v \in B$,} \\
                &|A|, &{}    & \text{if } (u,v)=(t,s).
            \end{aligned}
            \right. \\
            l(u,v)&=\left\{
            \begin{aligned}
                &1-rz, &\quad & \text{if $u=s,v \in A$ or $u \in B,v=t$,} \\
                &0,   &{}    & \text{if $u \in A$ and $v \in B$ or $(u,v)=(t,s)$,} \\
            \end{aligned}
            \right.
        \end{align*}
        We claim that $(z,\bm{a},\bm{h})$ is a feasible solution of (BLP2) if and only if $H$ has a feasible circulation.

        If $D$ has a feasible circulation $f$, then we set
        \begin{align*}
            h_v&=\left\{
            \begin{aligned}
                &\frac{1}{r}(1-f(s,v)), &\quad & \text{if } v \in A, \\
                &\frac{1}{r}(1-f(v,t)), &{}    & \text{if } v \in B, \\
            \end{aligned}
            \right. \\
            a_{uv}&=f(u,v) \text{ for all } u \in A \text{ and } v \in B.
        \end{align*}
        Then by the conservation condition, for every vertex $u \in A$, we have
        \begin{equation*}
            \bm{\partial}_G(u)\bm{a}+rh_u=\sum\limits_{v\in N^+_D(u)} f(u,v)+r \cdot \frac{1}{r}(1-f(s,u))=1-f(s,u)+\sum\limits_{v\in N^+_D(u)} f(u,v)=1.
        \end{equation*}
        By lower bound and capacity constraint, we can verify that $(z,\bm{a},\bm{h})$ satisfies the rest constraints of (BLP2). So $(z,\bm{a},\bm{h})$ is a feasible solution of (BLP2). 
        Conversely, if $(z,\bm{a},\bm{h})$ satisfies the constraints of (BLP2), then we define a function $f$ on $E(D)$ as follows:
        \begin{equation*}
            f(u,v)=\left\{
            \begin{aligned}
                &1-rh_u,                  &\quad & \text{if $u=s$ and $v \in A$,} \\
                &1-rh_v,                  &\quad & \text{if $u \in B$ and $v=t$,} \\
                &a_{uv},                 &{}    & \text{if $u \in A$ and $v \in B$,} \\
                &|B|-r\sum_{v \in B}h_v,  &\quad & \text{if } (u,v)=(t,s).
            \end{aligned}
            \right.
        \end{equation*}
        Clearly, $f$ is a feasible circulation. Thus, our claim holds.

        By Theorem \ref{circulation_theorem}, $(z,\bm{a},\bm{h})$ is a feasible solution of (BLP2) if and only if every $R \subseteq V(D)$ satisfies $c(V(D)\backslash R,R) \geqslant l(R,V(D) \backslash R)$. We set $X=A\backslash R$ and $Y=B \cap R$.

        \textbf{Case 1.} $s \notin R$ and $t \notin R$. Then
        \begin{align*}
            c(V(D) \backslash R,R)&=c(X,Y)+c(s,A \backslash X)=ze(X,Y)+|A \backslash X|, \\
            l(R,V(D) \backslash R)&=l(A \backslash X,B \backslash Y)+l(Y,t)=|Y|(1-rz).
        \end{align*}
        If $Y = \emptyset$, then $e(X,Y)=|Y|=0$, so $c(V(D) \backslash R,R) \geqslant l(R,V(D) \backslash R)$ holds.
        If $Y \neq \emptyset$, then $c(V(D)\backslash R,R) \geqslant l(R,V(D) \backslash R)$ if and only if $z \geqslant \frac{|X|+|Y|-|A|}{e(X,Y)+r|Y|}$.

        \textbf{Case 2.} $s \in R$ and $t \notin R$. Then
        \begin{align*}
            c(V(D) \backslash R,R)&=c(X,Y)+c(t,s)=ze(X,Y)+|A|, \\
            l(R,V(D) \backslash R)&=l(s,X)+l(A \backslash X,B \backslash Y)+l(Y,t)=|X|(1-rz)+|Y|(1-rz).
        \end{align*}
        If $X = \emptyset$ and $Y = \emptyset$, then $e(X,Y)=|X|=|Y|=0$, so $c(V(D) \backslash R,R) \geqslant l(R,V(D) \backslash R)$ holds.
        If $X \neq \emptyset$ or $Y \neq \emptyset$, then $c(V(D)\backslash R,R) \geqslant l(R,V(D) \backslash R)$ if and only if $z \geqslant \frac{|X|+|Y|-|A|}{e(X,Y)+r|X|+r|Y|}$.

        \textbf{Case 3.} $s \notin R$ and $t \in R$. Then
        \begin{align*}
            c(V(D) \backslash R,R)&=c(X,Y)+c(s,A \backslash X)+c(B \backslash Y,t)=ze(X,Y)+|A\backslash X|+|B \backslash Y|, \\
            l(R,V(D) \backslash R)&=l(A \backslash X,B \backslash Y)+l(t,s)=0.
        \end{align*}
        So $c(V(D)\backslash R,R) \geqslant l(R,V(D) \backslash R)$ holds for any $z$.

        \textbf{Case 4.} $s \in R$ and $t \in R$. Then
        \begin{align*}
            c(V(D) \backslash R,R)&=c(X,Y)+c(B \backslash Y,t)=ze(X,Y)+|B \backslash Y|, \\
            l(R,V(D) \backslash R)&=l(s,X)+l(A \backslash X,B \backslash Y)=|X|(1-rz).
        \end{align*}
        If $X = \emptyset$, then $e(X,Y)=|X|=0$, so $c(V(D) \backslash R,R) \geqslant l(R,V(D) \backslash R)$ holds.
        If $X \neq \emptyset$, then $c(V(D)\backslash R,R) \geqslant l(R,V(D) \backslash R)$ if and only if $z \geqslant \frac{|X|+|Y|-|A|}{e(X,Y)+r|X|}$.

        So $(z,\bm{a},\bm{h})$ is a feasible solution of (BLP2) if and only if for every $X \subseteq A$ and $Y \subseteq B$, $z \geqslant \max\left\{\frac{|X|+|Y|-|A|}{e(X,Y)+r|X|},\frac{|X|+|Y|-|A|}{e(X,Y)+r|Y|}\right\}$. Thus, we have
        \begin{equation*}
            L(G \square H)=\max\left\{\left. \frac{|X|+|Y|-|A|}{e(X,Y)+r\min\{|X|,|Y|\}}\right| X \subseteq A, Y \subseteq B \right\} >0,
        \end{equation*}
        and
        \begin{equation*}
            mp_f(G \square H)= \min \left\{ \left. \frac{e(X,Y)+r\min \{|X|,|Y|\}}{|X|+|Y|-|A|} \right| X \subseteq A, Y \subseteq B, |X|+|Y|-|A|>0 \right\}.
        \end{equation*}
    \end{proof}
    Then we obtain an inequality on fractional matching preclusion number of two bipartite graphs and their Cartesian product.
    \begin{The}
        \label{lower_bound_cartisian_product}
        Let $G$ and $H$ be two bipartite graphs. Then $mp_f(G \square H) \geqslant mp_f(G)+\lfloor mp_f(H) \rfloor$ and the equality holds when both $G$ and $H$ are regular.
    \end{The}
    \begin{proof}
        We suppose $G$ has a bipartition $(A,B)$ and $r=\lfloor mp_f(H) \rfloor$. Then $H$ has an $r$-factor $H'$ by Corollary \ref{mpstar_k_factor}. If $|A| \neq |B|$, then $G$ has no perfect matching, so $mp_f(G)=0$. Moreover, $G \square H$ also has an $r$-factor by Lemma \ref{cartesian_product_pm}. Thus, by Corollary \ref{mpstar_k_factor} we have
        \begin{equation*}
            mp_f(G \square H) \geqslant \lfloor mp_f(G \square H) \rfloor \geqslant r = mp_f(G)+ \lfloor mp_f(H) \rfloor.
        \end{equation*}

        Next we suppose that $|A|=|B|$. By Theorem \ref{mpstar_product}, we suppose there exist two vertex sets $X_0 \subseteq A$ and $Y_0 \subseteq B$ such that
        \begin{equation*}
            mp_f(G \square H')=\frac{e(X_0,Y_0)+r\min \{|X_0|,|Y_0|\}}{|X_0|+|Y_0|-|A|}.
        \end{equation*}
        So by Theorem \ref{mpstar}, we have
        \begin{align*}
            mp_f(G) &\leqslant \frac{e(X_0,Y_0)}{|X_0|+|Y_0|-|A|} \\
                    &=         mp_f(G \square H')-\frac{r\min \{|X_0|,|Y_0|\}}{|X_0|+|Y_0|-|A|} \\
                    &\leqslant mp_f(G \square H')-r \\
                    &=         mp_f(G \square H')-\lfloor mp_f(H) \rfloor.
        \end{align*}
        Since $E(G \square H') \subseteq E(G \square H)$, we have $\mathcal{M}(G \square H') \subseteq \mathcal{M}(G \square H)$, so the constraints in (FMP) for $G \square H'$ are also constraints in (FMP) for $G \square H$.
        Thus, we have $mp_f(G \square H) \geqslant mp_f(G \square H') \geqslant mp_f(G)+\lfloor mp_f(H) \rfloor$.

        Next we show the equality holds when both $G$ and $H$ are regular. If $G$ is $r_1$-regular and $H$ is $r_2$-regular, then $G \square H$ is $(r_1+r_2)$-regular. So by Theorem \ref{regular_bipartite}, $G$, $H$ and $G \square H$ have $r_1$-factor, $r_2$-factor and $(r_1+r_2)$-factor, respectively. By Corollary \ref{mpstar_k_factor}, we have $mp_f(G)=r_1$, $mp_f(H)=r_2$ and $mp_f(G \square H)=r_1+r_2$, then $mp_f(G \square H)=mp_f(G)+\lfloor mp_f(H) \rfloor$.
    \end{proof}


\begin{thebibliography}{99}
    \bibitem{f_factor_algorithm}R. Anstee, An algorithmic proof of Tutte's f-factor theorem, J. Algorithms 6 (1985) 112--131.
    \bibitem{bipartite_perfect_matching_polytope}G. Birkhoff, Tres observaciones sobre el algebra lineal, Rev. Univ. Nac. Tucum\'{a}n (Ser. A) 5 (1946) 147--151.
    \bibitem{matching_preclusion}R. C. Brigham, F. Harary, E. C. Violin, J. Yellen, Perfect matching preclusion, Congr. Numer. 174 (2005) 185--192.
    \bibitem{perfect_matching_polytope}J. Edmonds, Maximum matching and a polyhedron with $0,1$-vertices, J. Res. Nat. Bur. Standards 69B (1965) 125--130.
    \bibitem{folded_petersen_cube}E. Cheng, R. Connolly, C. Melekian, Matching preclusion and conditional matching preclusion problems for the folded Petersen cube, Theor. Comput. Sci. 576 (2015) 30--44.
    \bibitem{bipartite1}E. Cheng, P. Hu, R. Jia, and L. Lipt\'{a}k, Matching preclusion and conditional matching preclusion for bipartite interconnection networks I: Sufficient conditions, Networks, 59 (2012) 349--356.
    \bibitem{bipartite2}E. Cheng, P. Hu, R. Jia, L. Lipt\'{a}k, Matching preclusion and conditional matching preclusion problems for bipartite interconnection networks II: Cayley graphs generated by transposition trees and hyper-stars, Networks 59 (2012) 357--364.
    \bibitem{pancake1}E. Cheng, P. Hu, R. Jia, L. Lipt\'{a}k, B. Scholten, J. Voss, Matching preclusion and conditional matching preclusion for pancake and burnt pancake graphs, Int. J. Parallel Emergent Distrib. Syst. 29 (2014) 499--512.
    \bibitem{regular}E. Cheng, M.J. Lipman, L. Lipt\'{a}k, Matching preclusion and conditional matching preclusion for regular interconnection networks, Discrete Appl. Math. 160 (2012) 1936--1954.
    \bibitem{product}E. Cheng, L. Lipt\'{a}k, Matching preclusion and conditional matching preclusion problems for tori and related Cartesian products, Discrete Appl. Math. 12 (2012) 1699--1716.
    \bibitem{combiatorial_book}W.J. Cook, W.H. Cunningham, W.R. Pulleyblank, A. Schrijver, Combinatorial Optimization, Wiley, New York, 1998.
    \bibitem{r_factor2}J. Folkman and D. R. Fulkerson, Flows in infinite graghs, J. Combin. Theory 8 (1970) 30--44.
    \bibitem{max_flow_min_cut}L. R. Ford, D. R. Fulkerson, Maximal flow through a network, Canad. J. Math. 8 (1956) 399--404.
    \bibitem{eq_sep_opt}M. Gr\"{o}tschel, L. Lov\'{a}sz, A. Schrijver, Geometric Algorithms and Combinatorial Optimization, Springer-Verlag, Berlin, 1988.
    \bibitem{hall}P. Hall, On representatives of subsets, J. London Math. Soc. 10 (1935) 26--30.
    \bibitem{circulation_theorem}A. J. Hoffman, Some recent applications of the theory of linear inequalities to extremal combinatorial analysis, Proc. Symp. Appl. Math., Vol. X (1960) pp. 113--127.
    \bibitem{pancake2}X. Hu, H. Liu, The (conditional) matching preclusion for burnt pancake graphs, Discrete Appl. Math. 161 (2013) 1481--1489.
    \bibitem{NPC_MP}M. Lacroix, A.R. Mahjoub, S. Martin, C. Picouleau, On the NP-completeness of the perfect matching free subgraph problem, Theor. Comput. Sci. 423 (2012) 25--29.
    \bibitem{vertex_transitive}Q. Li, J. He, H. Zhang, Matching preclusion for vertex-transitive networks, Discrete Appl. Math. 207 (2016) 90--98.
    \bibitem{cube_connected_cycle}Q. Li, W. C. Shiu, H. Yao, Matching preclusion for cube-connected cycles, Discrete Appl. Math. 190--191 (2015) 118--126.
    \bibitem{LP_of_MP}R. Lin, H. Zhang, Maximally matched and super matched regular graphs, Int. J. Comput. Math.: Comput. Syst. Theory 2 (2016) 74--84.
    \bibitem{fractional_matching_preclusion}Y. Liu, W. Liu, Fractional matching preclusion of graphs, J. Comb. Optim. (2016), http://dx.doi.org/10.1007/s10878-016-0077-x.
    \bibitem{balanced_hypercube}H. L\"{u}, X. Li, H. Zhang, Matching preclusion for balanced hypercubes, Theor. Comput. Sci. 465 (2012) 10--20.
    \bibitem{r_factor1}O. Ore, Graphs and subgraphs, Trans. Amer. Math. Soc. 84 (1957) 109--136.
    \bibitem{odd_cut}M.W. Padberg, M.R. Rao, Odd minimum cut-sets and b-matchings, Math. Oper. Res. 7 (1982) 67--80.
    \bibitem{fractional_graph_theory}E.R. Scheinerman, D.H. Ullman, Fractional Graph Theory: A Rational Approach to the Theory of Graphs, John Wiley, New York, 1997.
    \bibitem{k_ary_n_cube}S. Wang, R. Wang, S. Lin, J. Li, Matching preclusion for $k$-ary $n$-cubes, Discrete Appl. Math. 158 (2010) 2066--2070.
\end{thebibliography}

\end{document}